\documentclass[a4paper,11pt]{article}

\usepackage{amsmath}
\usepackage{amssymb}
\usepackage{amsthm}
\usepackage{graphicx}
\usepackage{caption}
\usepackage{subcaption}
\usepackage{psfrag}
\usepackage{verbatim}
\usepackage{color}
\usepackage{mathtools}
\usepackage[%
ocgcolorlinks,%
linkcolor=blue,%
filecolor=blue,%
citecolor=blue,%
urlcolor=blue]{hyperref}

\newcommand{\bbM}{\mathbb{M}}
\newcommand{\bbE}{\mathbb{E}}
\newcommand{\bbV}{\mathbb{V}}

\newcommand{\Var}{\bbV{\rm ar}}
\newcommand{\bbP}{\mathbb{P}}

\newcommand{\bbR}{\mathbb{R}}

\newcommand{\wh}{\widehat}

\newcommand{\cA}{\mathcal A}

\newcommand{\cC}{\mathcal C}

\newcommand{\cB}{\mathcal B}

\newcommand{\cE}{\mathcal E}
\newcommand{\cN}{\mathcal N}
\newcommand{\cF}{\mathcal F}

\newcommand{\cR}{\mathcal R}

\newcommand{\rmc}{{\rm c}}

\newcommand{\rmd}{{\rm d}}
\newcommand{\rme}{{\rm e}}
\newcommand{\rmB}{{\rm B}}

\newcommand{\wt}{\widetilde}

\newtheorem{theorem}{Theorem}[section]

\newtheorem{lem}[theorem]{Lemma}
\newtheorem{prop}[theorem]{Proposition}
\newtheorem{thm}[theorem]{Theorem}
\newtheorem{cor}[theorem]{Corollary}

\numberwithin{equation}{section}

\title{More on the Structure of Extreme Level Sets in Branching Brownian Motion}
\author
{Aser Cortines\thanks{aser.cortinespeixoto@math.uzh.ch, lhartung@uni-mainz.de,  oren.louidor@gmail.com.} \\Universit\"at Z\"urich \and Lisa Hartung\footnotemark[1]\\ 
Universit\"at Mainz\and Oren Louidor\footnotemark[1]\\Technion, Israel}
\date{}

\begin{document}

\maketitle

\begin{abstract}
This work is a continuation of~\cite{CHL17}, in which the same authors studied the fine structure of the extreme level sets of branching Brownian motion, namely the sets of particles whose height is within a finite distance from the global maximum. It is well known that such particles congregate at large times in clusters of order-one genealogical diameter around local maxima which form a Cox process in the limit. Our main finding here is that most of the particles in an extreme level set come from only a small fraction of the clusters, which are atypically large.
\end{abstract}

\section{Introduction and Results}
\subsection{Setup and state of the art}
\label{ss:Introduction}
This work is a continuation of~\cite{CHL17}, in which the fine structure of the extreme values of branching Brownian motion (BBM) was studied. Let us first recall the definition of BBM and some of the state-of-the-art concerning its extreme value statistics. Let $L_t$ be the set of particles alive at time $t \geq 0$ in a continuous time Galton-Watson process with binary branching at rate $1$. The entire genealogy can be recorded via the metric space $(T,\rmd)$, consisting of the elements $T := \cup_{t \geq 0} L_t$ and equipped with the {\em genealogical distance},
\begin{equation}
\nonumber
\rmd (x,x') := 
\inf \Big\{\tfrac{(t-s) + (t'-s)}{2} :\: s \geq 0, \ x,x' \text{ share a common ancestor in } L_s \Big\}
\quad ; x \in L_t,\, x' \in L_{t'} \,,
\end{equation}
for any $t, t' \geq 0$. 

Conditional on $(T,\rmd)$, let $h = (h(x) :\: x \in T)$ be a mean-zero Gaussian process with covariance function given by $\bbE h(x) h(x') = (t + t')/2 - d(x,x')$ for $x \in L_t$ and $x' \in L_{t'}$ and $t,t' \geq 0$. Equivalently, $\bbE h(x) h(x')$ is equal to the largest $s \geq 0$ such that both $x$ and $x'$ share a common ancestor at time $s$. Then the triplet $(h, T,\rmd)$ (or just $h$ for short) forms a standard BBM and $h(x)$ for $x \in L_t$ is interpreted as the height of particle $x$ at time $t$. The restriction of $T$ to all particles born up-to time $t$, will be denoted by $T_t := \cup_{s \leq t} L_s$, with $\rmd_t$ and $h_t$ the corresponding restrictions of $\rmd$ and $h$, respectively. The natural filtration of the process $(\cF_t :\: t \geq 0)$ can then be defined via $\cF_t = \sigma(h_t, T_t, \rmd_t)$ for all $t \geq 0$.

The study of extreme values of $h$ dates back to works of Ikeda et al.~\cite{Watanabe1,Watanabe2,Watanabe3}, McKean~\cite{McKean}, Bramson \cite{B_C, Bramson1978maximal} and Lalley and Sellke~\cite{LS} who derived asymptotics for the law of the
maximal height $h^*_t = \max_{x \in L_t} h_t(x)$. Introducing the centering function
\begin{equation}
\label{eq_m_t}
m_t := \sqrt{2}t - \frac{3}{2\sqrt{2}} \log^+ t \,,
\quad  \text{where } \qquad
\log^+ t := \log (t \vee 1) \,,
\end{equation}
and writing
$\wh{h}_t$ for the centered process $h_t - m_t$ and $\wh{h}^*_t : = h_t^\ast - m_t$ for its maximum, these works show that
$\wh{h}^*_t$ converges in law to $G + (1/\sqrt{2})\log Z$ as $t \to \infty$,
where $G$ is a Gumbel distributed random variable and $Z$, which is independent of $G$, is the almost sure limit as $t \to \infty$ of (a multiple of) the so-called {\em derivative martingale}:
\begin{equation}
\label{e:303}
\textstyle
Z_t := C_\diamond \sum_{x \in L_t} \big( \sqrt{2} t - h_t(x) \big) \rme^{\sqrt{2} (h_t(x) - \sqrt{2}t)} \,,
\end{equation}
for some properly chosen $C_\diamond > 0$.

Other extreme values of $h$ can be studied simultaneously by considering the {\em extremal process} associated with it. To describe the latter, given $t \geq 0$, $x \in L_t$ and $r > 0$, we let $\cC_{t,r}(x)$ denote the {\em cluster} of {\em relative heights} of particles in $L_t$, which are at genealogical distance at most $r$ from $x$. This is defined formally as the point measure,
\begin{equation}
\label{e:5B}
\textstyle
\cC_{t,r}(x) := \sum_{y \in \rmB_{r}(x)} \delta_{h_t(y) - h_t(x)},\  \text{ where} \quad
\rmB_{r}(x) := \{y \in L_t :\: \rmd (x,y) < r\} \,.
\end{equation}

Fixing any positive function $t \mapsto r_t$ such that both $r_t$ and $t-r_t$ tend to $\infty$ as $t \to \infty$ and letting $L_t^* = \big\{x \in L_t :\: h_t(x) \geq h_t(y) \,, \forall y \in \rmB_{r_t} (x)\big\}$, the {\em structured extremal process} is then given as
\begin{equation}
\label{e:N6}
\textstyle
\wh{\cE}_t := \sum_{x \in L_t^*} \delta_{h_t(x) - m_t} \otimes \delta_{\cC_{t,r_t}(x)} \,.
\end{equation}
That is, $\wh{\cE_t}$ is a point process on $\bbR \times \bbM$, where $\bbM$ denote the space of all point measures on $(-\infty, 0]$, which records {\em the-centered-height-of} ($u$) and {\em the-cluster-around} ($\cC$) -- all $r_t$-local maxima of $h$. We shall sometimes refer to the pair $(u, \cC)$ as a {\em cluster-pair}. It was then shown in~\cite{ABBS, ABK_E} that
\begin{equation}
\label{e:N7}
\big(\wh{\cE}_t, Z_t \big) \overset{t \to \infty}{\Longrightarrow} \big(\wh{\cE}, Z\big)\,,
\quad \text{where}\qquad 
\wh{\cE}|Z \sim {\rm PPP}(Z\rme^{-\sqrt{2}u} \rmd u \otimes \nu) \,.
\end{equation}
Above $Z_t$ and $Z$ are as before and $\nu$ is a deterministic distribution on $\bbM$, which we will call the {\em cluster distribution}. As two consequences, one gets the convergence of the {\em standard} extremal process of $h$:
\begin{equation}
\label{e:O.2}
\cE_t \overset{t \to \infty}\Longrightarrow \cE 
\quad  \text{where} \quad
\cE_t := \sum_{x \in L_t} \delta_{h_t(x) - m_t} 
\ , \quad 
\cE
:= \sum_{(u, \cC) \in \wh{\cE}} \, \cC(\cdot - u) \,,
\end{equation}
as well as the convergence of the {\em extremal process of local maxima}:
\begin{equation}
\label{e:5.5}
\cE^*_t \overset{t \to \infty}\Longrightarrow \cE^* 
\quad \text{where} \quad
\cE_t^* := \sum_{x \in L^*_t} \delta_{h_t(x) - m_t} 
\ , \quad 
\cE^* := \sum_{(u, \cC) \in \wh{\cE}} \, \delta_u 
\sim {\rm PPP}(Z\rme^{-\sqrt{2}u}) \,.
\end{equation}
Henceforth, we shall use the unified notation $\cE_{(t)}$ (also $\cE^*_{(t)}$, $\wh{\cE}_{(t)}$) to mean $\cE$ or $\cE_t$.

The asymptotic growth of the number of extreme values which are also $r_t$-local maxima can then be read off directly from~\eqref{e:5.5}. Indeed, a simple application of the weak law of large numbers combined with the convergence statement in~\eqref{e:5.5} yield,
\begin{equation}
\label{e:1.8}
\frac{\cE^*([-u, \infty))}{Z \rme^{\sqrt{2} u}/\sqrt{2}} 
	\overset{\bbP} \longrightarrow  1
\text{ as } u \to \infty 
\quad, \qquad 
\frac{\cE^*_t([-u, \infty))}{Z \rme^{\sqrt{2} u}/\sqrt{2}} \overset{\bbP}\longrightarrow 1
\text{ as } t \to \infty \text{ then } u \to \infty  \,.
\end{equation}
The asymptotic growth of the number of {\em all} extreme values, which is arguably the more interesting quantity, is not however a straightforward consequence of~\eqref{e:O.2}. This is because the limiting process $\cE$ is now a superposition of i.i.d. clusters $\cC$, and the law $\nu$ of the latter will determine the number of points inside any given set in the overall process. 

To address this question, a study of the cluster law $\nu$ was carried out in~\cite{CHL17} and then used to show (Proposition~1.5) that 
\begin{equation}
\label{e:29}
\bbE \cC([-v, 0]) \sim C_\star \rme^{\sqrt{2} v} 
\ \text{as } v \to \infty 
\quad ; \quad \cC \sim \nu \,,
\end{equation}
for some $C_\star > 0$. This was then combined with~\eqref{e:N7} and~\eqref{e:O.2} to derive
(Theorem~1.1 in~\cite{CHL17})
\begin{equation}
\label{e:1.10}
\frac{\cE([-v, \infty))}{C_\star Z v \rme^{\sqrt{2} v}} \overset{\bbP}{\longrightarrow} 1
\text{ as } v \to \infty 
\quad, \qquad
\frac{\cE_t([-v, \infty))}{C_\star Z v \rme^{\sqrt{2} v}}
\overset{\bbP}{\longrightarrow} 1
\text{ as } t \to \infty \text{ then } v \to \infty \,,
\end{equation}
The above shows that points coming from the clusters around extreme local maxima account for an additional multiplicative linear prefactor in the overall growth rate of extreme values. 

Next, it is natural to ask what are the {\em typical height} $u \in \bbR$ and {\em typical cluster configuration} $\cC$ of those cluster-pairs $(u, \cC) \in \wh{\cE}_{(t)}$ which {\em carry} the level set $\cE_{(t)}|_{[-v, \infty)}$. The question of the typical height was addressed in~\cite{CHL17}. To give a precise statement, for a Borel set $B \subseteq \bbR \times \bbM$ let us define (compare with~\eqref{e:O.2}):
\begin{equation}
\label{e:421}
\cE_t (\cdot \;; B) := \sum\limits_{(u,\cC) \in \wh{\cE}_t} \cC ( \cdot - u) 1_{\{(u, \cC) \in B\}} \quad \text{and} \quad
\cE (\cdot \;; B) := \sum\limits_{(u,\cC) \in \wh{\cE}} \cC ( \cdot - u) 1_{\{(u, \cC) \in B\}} \\
\,.
\end{equation}
These processes record all extreme values coming from cluster pairs $(u, \cC)$ in $B$.  Theorem~1.2 from~\cite{CHL17} then showed
that for any $\alpha \in (0,1]$, as $v \to \infty$
\begin{equation}
\label{e:37}
\frac{\cE\big([-v,\infty) ;\; [-\alpha v, \infty) \times \bbM \big)}{
\cE \big([-v,\infty)\big)} \underset{v \to \infty}{\overset{\bbP}\longrightarrow} \alpha 
\quad , \qquad
\frac{\cE_t\big([-v,\infty) ;\; [-\alpha v, \infty) \times \bbM \big)}{
	\cE_t \big([-v,\infty)\big)} \underset{\substack{t \to \infty\\ v \to \infty}}{\overset{\bbP}\longrightarrow} \alpha \,.
\end{equation}
In other words, the typical height $u$ of those extreme local maxima of which clusters contribute to $\cE_{(t)}|_{[-v, \infty)}$ is uniformly chosen in $\big[-v, O(1)\big)$.
This left the open question of the typical cluster configurations carrying $\cE_{(t)}|_{[-v, \infty)}$, and this is precisely the focus of this manuscript. 

\subsection{New results}
Our first observation is that $\cC([-v,0])$ is not concentrated around its mean. This could be readily concluded from the results in~\cite{CHL17}, by comparing the first and second moment of this quantity (see also Lemma~\ref{l:12} below). The next proposition provides a quantitative version for this assertion.
\begin{prop}
\label{p:2.3}
Let $\cC \sim \nu$. For all $\epsilon > 0$, there exists $\delta > 0$, such that for all $v$ large enough,
\begin{equation}
\label{e:33.2}
\bbE \Big( \cC([-v,0]) ;\; \cC([-v,0]) \leq \delta v \rme^{\sqrt{2} v} \Big)
\leq \epsilon \rme^{\sqrt{2} v} \,.
\end{equation}
Moreover, there exists $C > 0$ such that for all $\delta > 0$ and $v > 0$,
\begin{equation}
\label{e:33.1}
\bbP \Big( \cC([-v,0]) >  \delta v \rme^{\sqrt{2} v} \Big) 
	\leq \frac{C}{\delta} \,v^{-1} \,.
\end{equation}
\end{prop} 
The above proposition shows that the asymptotic mean on the right hand side of~\eqref{e:29} is the result of an unlikely event of probability $O(v^{-1})$, in which the number of cluster points above $-v$ is of unusually high order $v \rme^{\sqrt{2} v}$. Let us call a cluster satisfying the event in~\eqref{e:33.1} a {\em $\delta$-fat cluster for the height $-v$}, or a {\em $(-v, \delta)$-fat} cluster, for short. 

Since $\cE_{(t)}$ is a superposition of many i.i.d. clusters $\cC$, the above proposition should translate, by virtue of the law of large numbers, to the assertion that most extreme values come from only a small fraction of the clusters -- the fat ones. To phrase this assertion in formal terms, if $v \geq 0$ and $\delta > 0$, we set
\begin{equation}
\label{e:127}
F_{\delta}(v) := \Big\{ (u, \cC) \in \bbR \times \bbM :\:
	u \geq v \,,\,\, \cC \text{ is $(v-u,\delta)$-fat} \Big \}\,.
\end{equation}

The cluster $\cC$ in the pair $(u,\cC) \in F_{\delta}(v)$ is fat precisely for the height that is relevant for its contribution to $\cE|_{[v, \infty)}$. Given a Borel set $B \subseteq \bbR \times \bbM$ we also define in analog to~\eqref{e:421},
\begin{equation}
\label{e:421.5}
\cE^*_t (\cdot \;; B) := \sum_{(u,\cC) \in \wh{\cE}_t} \delta_u 1_{\{(u,\cC) \in B\}} \,.
\quad \text{and} \quad
\cE^* (\cdot \;; B) := \sum_{(u,\cC) \in \wh{\cE}} \delta_u 1_{\{(u,\cC) \in B\}}
\end{equation}
Then,
\begin{thm}
\label{t:15}
For all $\epsilon > 0$ there exist $\delta >0$, such that for all $\alpha \in (0,1)$ with probability tending to $1$ as $v \to \infty$, 
\begin{equation}
\label{e:159}
	\frac{\cE \big([-v, \infty) ;\;  [-\alpha v, \infty) \times \bbM ,\,  F_\delta(-v) \big)}
	{\cE\big([-v, \infty) ;\; [-\alpha v, \infty) \times \bbM \big)} \geq 1-\epsilon \,.
\end{equation}
Moreover, there exists $C > 0$ such that for all $\delta > 0$ and $\alpha \in (0,1)$, with probability tending to $1$ as $v \to \infty$,
\begin{equation}
\label{e:158}
	\frac{\cE^*\big([-\alpha v, \infty) ;\; F_\delta(-v)\big)}
		{\cE^*\big([-\alpha v, \infty)\big)} < \frac{C}{\delta(1-\alpha)} \,v^{-1} \,.
\end{equation}
Both statements hold in the limit when $t \to \infty$ followed by $v \to \infty$ if we replace $\cE$ and $\cE^*$ by $\cE_t$ and $\cE^*_t$ respectively.
\end{thm}
As a corollary we get,
\begin{cor} \label{c:1.3}
Let $\epsilon > 0$ be arbitrarily small. Then for all $v$ large enough, we may find $G_\epsilon(-v) \subseteq \bbR \times \bbM$, such that with probability at least $1-\epsilon$, 
\begin{equation}
\label{e:158a}
	\frac{\cE \big([-v, \infty) ;\;  G_\epsilon(-v) \big)}
	{\cE\big([-v, \infty)\big)}  \geq 1-\epsilon \,.
\qquad \text{and} \qquad\ 
	\frac{\cE^*\big([-v, \infty) ;\; G_\epsilon(-v)\big)}
		{\cE^*\big([-v, \infty)\big)} < \epsilon 
\end{equation}
The same holds for all $t$ large enough (depending on $v$ and $\epsilon$), if we replace $\cE$ and $\cE^*$ by $\cE_t$ and $\cE^*_t$ respectively.
\end{cor}

\subsection{Proof idea and heuristic picture}
The proof of Theorem~\ref{t:15} is a rather standard application of Chebyshev's inequality, using Proposition~\ref{p:2.3} (along with a second moment bound on $\cC([-v,0])$ from~\cite{CHL17}) and the explicit (conditional) Poissonian structure of $\wh{\cE}$. We therefore omit further explanations of this proof and the straightforward one for Corollary~\ref{c:1.3}, focusing instead on the argument for Proposition~\ref{p:2.3}. As in~\cite{CHL17}, the key ingredient is a handle on the cluster distribution $\nu$. This was carried out in Section~3 of~\cite{CHL17} and is presented in this manuscript again in Section~\ref{s:Reduction}. Let us therefore first briefly describe this derivation and recall how it was used in~\cite{CHL17} to show~\eqref{e:29}. The reader is referred to~\cite{CHL17} for more details.

Thanks to~\eqref{e:N7} and the product structure of the intensity measure in the definition of $\wh{\cE}$, the law $\nu$ can be obtained as the limiting distribution of the cluster around 
a uniformly chosen particle $X_t$ in $L_t$, conditioned to be the global maximum at time $t$ and having height, say, $m_t$. Tracing the trajectory of this distinguished particle backwards in time and accounting, via the spinal decomposition (Many-to-one Lemma, see Subsection~\ref{ss:Spine}), for the random genealogical structure, one sees a particle performing 
a standard Brownian motion $W=(W_s)_{s \geq 0}$ from $m_t$ at time $0$ to $0$ at time $t$. This, so-called, {\em spine particle} gives birth at random Poissonian times (at an accelerated rate $2$, see Subsection~\ref{ss:Spine}) to independent standard branching Brownian motions, which then evolve back to time $0$ and are conditioned to have their particles stay below $m_t$ at this time. The cluster distribution at genealogical distance $r$ around $X_t$ is therefore determined by the relative heights of particles of those branching Brownian motions which branched off before time $r$ (see Figure~\ref{f:Spinal}).

\begin{figure}
\centering
\begin{minipage}[b]{.45\textwidth}
\centering
    \includegraphics[width=\textwidth]{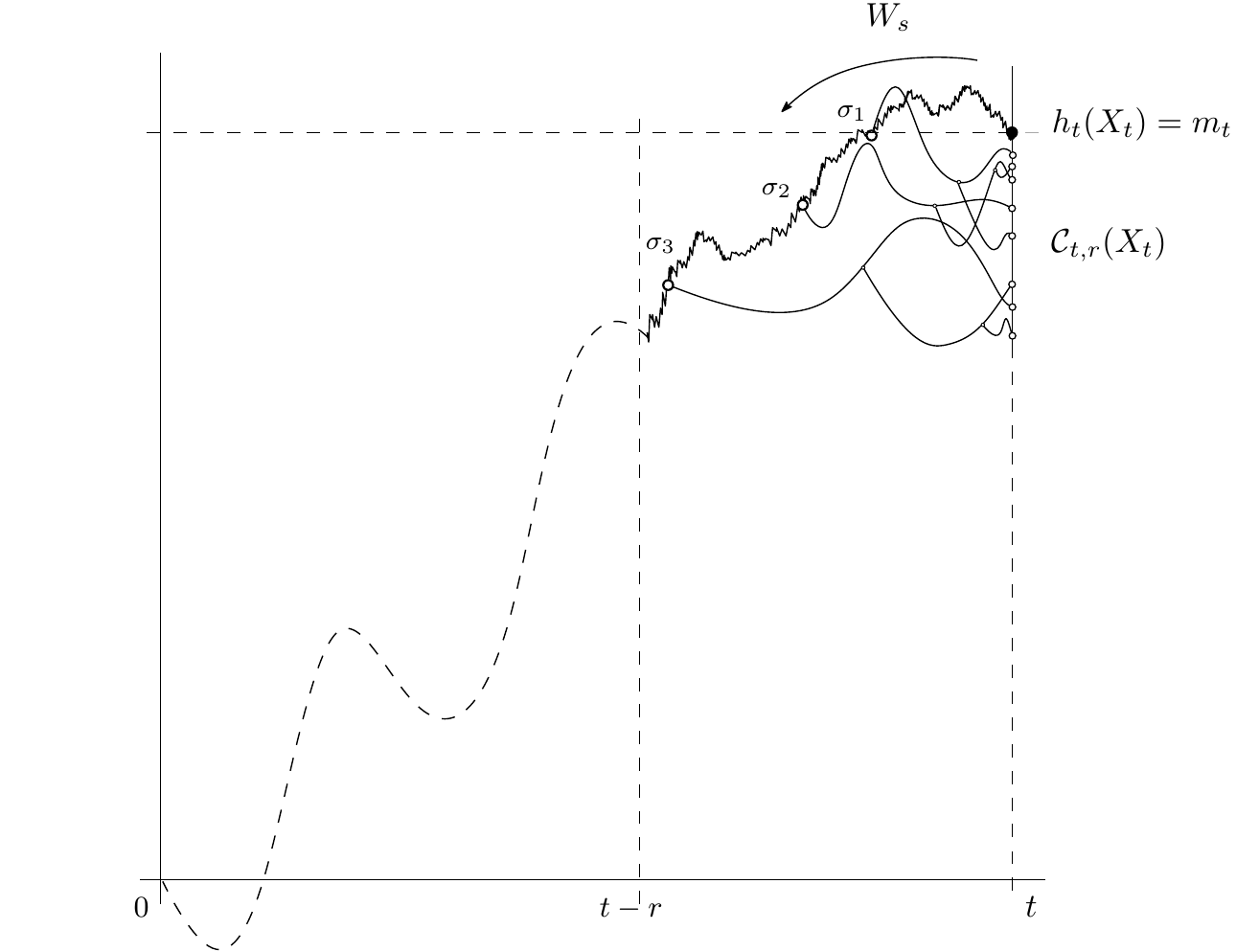} 
    \caption{\small The cluster $\cC_{t,r}(X_t)$ around the spine $X_t$, conditioned to be the maximum and at height $m_t$. The process $W_s$ is a Brownian bridge from $(0,m_t)$ to $(t,0)$ and $\sigma_1, \sigma_2, \dots$ are the branching times.
    } \label{f:Spinal}
\end{minipage}%
\hfill
\begin{minipage}[b]{.45\textwidth}
\centering
    \includegraphics[width=\textwidth]{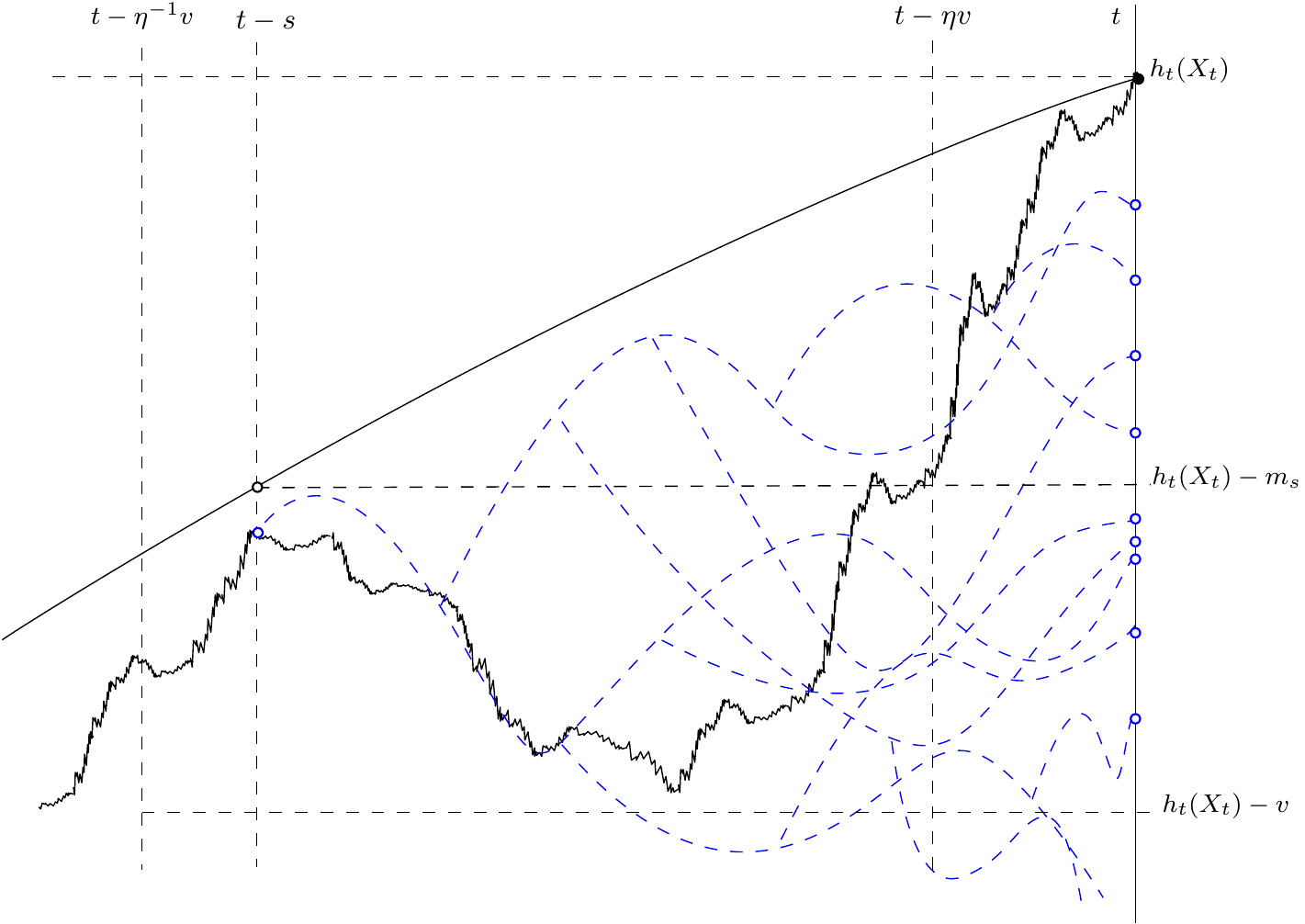} 
    \caption{\small The contribution to the cluster level set $\cC_{t,r_t}(X_t)|_{(-v,0]}$ on   
    an atypical event when the cluster maximum $X_t$ ascends from 
$h_t(X_t)-m_s+O(1)$ at time $t-s$ to $h_t(X_t)$ at time $t$, where $s \in [\eta v^2, \eta^{-1} v^2]$.}
    \label{f:fatcluster}
\end{minipage}

\end{figure}


Formally, denoting by $0 \leq \sigma_1 < \sigma_2 < \dots$ the points of a Poisson point process $\cN$ on $\bbR_+$ with rate $2$ and letting $H = (h^{s}_t(x) :\: t \geq 0\,,\,\, x \in L^s_t)_{s \geq 0}$ be a collection of independent branching Brownian motions (with $W, \cN$ and $H$ independent), the cluster distribution $\nu$ can then be written as the weak limit (Lemma~\ref{l:7.0}):
\begin{equation}
\label{e:11}
\begin{split}
\nu(\cdot) & = \lim_{t \to \infty}
\bbP \big(\cC_{t,r_t} (X_t) \in \cdot \big|\, h^*_t = h_t(X_t) = m_t \big) \\
& =\lim_{t \to \infty} \bbP \! \Big( \! \sum\limits_{\sigma_k \leq r_t} \sum\limits_{x \in L^{\sigma_k}_{\sigma_k}} \! \delta_{h^{\sigma_k}_{\sigma_k}(x)+ W_{\sigma_k} - m_t}  \! \in \! \cdot
	 \ \Big| \max\limits_{k: \sigma_k \in [0,t]}  \big(W_{\sigma_k} \! +  h^{\sigma_k *}_{\sigma_k}\big) \! \leq  \! m_t ,\, W_0 = m_t,\, W_t = 0 \Big) \,,
\end{split}
\end{equation}
where $r_t$ is as in~\eqref{e:N6} and $h^*_t$ is as above~\eqref{eq_m_t}. Tilting by $s \mapsto -m_t(1-s/t)$ and setting $\wh{W}_{t,s} := W_s - \gamma_{t,s}$, with $\gamma_{t,s} := m_t (s/t) - m_s = 3/(2\sqrt{2})(\log^+ s - \tfrac{s}{t}\log^+ t)$, we obtain,
\begin{equation}
\label{e:12}
\nu(\cdot) = 
\lim_{t \to \infty} \bbP
\Big( \! \sum\limits_{\sigma_k \leq r_t} \, \cE_{\sigma_k}^{\sigma_k} \big(\cdot - \wh{W}_{t,\sigma_k}\big)
	\!  \in \! \cdot 
	 \ \Big|  \max\limits_{k : \sigma_k \in [0,t]} \big(\wh{W}_{t,{\sigma_k}}  \! +   \wh{h}^{\sigma_k*}_{\sigma_k}\big) \! \leq \!  0 ,\, \wh{W}_{t,0} = \wh{W}_{t,t} = 0 \Big)\,,
\end{equation}
where $\cE^{s}_t$ is the extremal process associated with $h^{s}_t$ and $\wh{h}^s_{t} = h^s_t - m_t$. In this paper we refer to the triplet $(\wh{W}, \cN, H)$ as a decorated random-walk-like (DRW) process (see Subsection~\ref{ss:2.2}).

Now let $\cC \sim \nu$ and pick $v > 0$. Provided that we can exchange limit and integration, it follows form~\eqref{e:12} and Palm calculus that $\bbE \cC([-v, 0])$ is equal to the limit as $t \to \infty$ of
\begin{equation}
\label{e:M1.34.1}
\frac{\displaystyle \int_{0}^{r_t} \! 2\rmd s \!\!\int \! \bbE \Big(
\cE^s_s \big([-v, 0] \!-\! z\big) ;\; z + \wh{h}^{s*}_s \leq 0 \Big)
\bbP \Big(
\max\limits_{\sigma_k \leq t} \big(\wh{W}_{t,{\sigma_k}} \!+\! \wh{h}^{\sigma_k*}_{\sigma_k}\big) \leq 0, \wh{W}_{t,s} \in \rmd z \Big| \wh{W}_{t,0}\! =\! \wh{W}_{t,t} \!=\! 0 \Big) }
{\bbP \Big(\max\limits_{k : \, \sigma_k \in [0,t]} \big(\wh{W}_{t,{\sigma_k}} + \wh{h}^{\sigma_k*}_{\sigma_k}\big) \leq 0 \Big| \wh{W}_{t,0} = \wh{W}_{t,t} = 0 \Big)} \,,
\end{equation}
where the inner integral is over $z = O(1)$ and is the result of the total probability formula, after conditioning on $\{\wh{W}_{t,s} = z\}$.

The left most term in the integrand is the first moment of the size of the (global) extreme level set of $h^s_s$, subject to a truncation event restricting the height of its global maximum. This was estimated in~\cite{CHL17} (Lemma 4.2 with $|z| = O(1)$ and $v \to \infty$) as
\begin{equation}
\label{e:M1.41}
\bbE \big(\cE_t\big([-v, 0]-z\big) ;\; \wh{h}^*_t \leq -z \big) \approx C v \rme^{\sqrt{2} v - v^2/(2t)} \,,
\end{equation}
The remaining terms are of the form
\begin{equation}
\label{e:1.23}
\bbP \big( \max_{k: \sigma_k \in [s_1,s_2]} \big(\wh{W}_{t,\sigma_k} + \wh{h}^{\sigma_k*}_{\sigma_k}\big) \leq 0 \,\Big|\, \wh{W}_{t,s_1} = x ,\, \wh{W}_{t,s_2} = y \Big) \,,
\end{equation}
namely the probability that a DRW process stays negative on $[s_1, s_2]$.

For standard Brownian motion, the well known reflection principle gives
\begin{equation}
\bbP \Big(\max_{s \in [s_1,s_2]} W_s \leq 0 \,\Big|\, W_{s_1}=x,\, W_{s_2} =y \Big) \sim \frac{2xy}{s_2-s_1} 
\quad \text{as }\  s_1 - s_1 \to \infty\,,
\end{equation}
It was shown in~\cite{CHL17} (Subsection 2.1; see also Lemma~\ref{lem:15} here), that similar estimates hold for~\eqref{e:1.23} as well. (This is because the drift function $\gamma_{t,s}$ is bounded by $1 + \log^+ (s \wedge (t-s))$, the random decorations $(h^{s*}_{s} :\: s \geq 0)$ are (at least) exponentially tight and the random sampling times $(\sigma_k :\: k \geq 1)$ arrive at a Poissonian rate.) We can therefore estimate the probability in the denominator in~\eqref{e:M1.34.1}
 by $C t^{-1}$ and the probability in the numerator by 
\begin{equation}
C z^2 (s(t-s))^{-1} \bbP(\wh{W}_{t,s} \in \rmd z | \wh{W}_{t,0} = \wh{W}_{t,t} = 0) \approx C' t^{-1} s^{-3/2} z^2 \,.
\end{equation}
Using also~\eqref{e:M1.41} and performing the integrating over $z = O(1)$ in~\eqref{e:M1.34.1}, we obtain,
\begin{equation}
\label{e:M1.42}
\bbE \cC([-v, 0]) \approx C v \rme^{\sqrt{2}v}
 \int_{s=0}^{\infty} s^{-3/2} \rme^{-v^2/(2s)} \rmd s 
= C \rme^{\sqrt2{v}} \int_{r=0}^{\infty} r^{-3/2} \rme^{-1/(2r)} \rmd r 
= C_\star \rme^{\sqrt2{v}} \,.
\end{equation}
This explains~\eqref{e:29} and is precisely how these asymptotics are derived in~\cite{CHL17} (Proposition 1.5).

\medskip
Now, looking more carefully at~\eqref{e:M1.42}, we see that for large $v$, ``most'' of the contribution to the first integral comes from $s \in [\eta v^2 \,,\, \eta^{-1} v^2]$. Rolling back the derivation all the way to~\eqref{e:M1.34.1}, the latter is equivalent to saying that most of the contribution to $\bbE \cC([-v, 0])$ comes from trajectories in which $\textstyle \max_{s \in [\eta v^2, \eta^{-1} v^2]} \big|\wh{W}_{t,s}\big| = O(1)$. Such trajectories have very small probability:
\begin{equation}
\nonumber
\bbP \Big(\textstyle \max\limits_{s/v^2 \in [\eta, \frac{1}{\eta}]} \big|\wh{W}_{t,s}\big| = O(1)
\,\,\Big|\,\max\limits_{\sigma_k \leq t} \big(\wh{W}_{t,{\sigma_k}} + \wh{h}^{\sigma_k*}_{\sigma_k}\big) \leq 0, \wh{W}_{t,0} = \wh{W}_{t,t} = 0 \Big)
\approx \displaystyle \int_{s=\eta v^2}^{v^2/\eta} C s^{-3/2} \rmd s = \frac{C'}{v}
\end{equation}
But when they occur, it follows from~\eqref{e:M1.41} with $s \in [\eta v^2, \eta^{-1} v^2]$ (and an additional concentration argument) that 
$\cC([-v, 0]) \approx C \cE^{s}_{s}\big([-v ,0] - z\big)1_{\{\wh{h}^{s*}_s \leq -z\}}
 \approx C' v \rme^{\sqrt{2}v} $.

Reversing further the two reduction steps~\eqref{e:11} and~\eqref{e:12}, we can rephrase the last statement as follows (see Lemma~\ref{l:7.6} and Lemma~\ref{l:7.7}): Having $C_\star \rme^{\sqrt{2}v}$ mean number of particles above $-v$ (for large $v$) in a cluster is the result of an atypical $O(v^{-1})$ probability event, in which the number of such particles is of order $v \rme^{\sqrt{2}v}$. Such a cluster is realized at a large time $t$ when its local maximum $X_t$ has ascended (atypically slowly) to $h_t(X_t)$ from $h_t(X_t)-m_s+O(1)$ at time $t-s$, where $s \in [\eta v^2, \eta^{-1} v^2]$ (see Figure~\ref{f:fatcluster}). 

\subsection*{Remainder of the paper}
Section~\ref{s:Reduction} makes rigorous the reduction of the study of $\nu$ to that of a DRW process conditioned to stay negative, as just described. It also provides the necessary probability estimates for the latter as well as needed moment bounds for $\cC([-v,0]$. Section~\ref{s:3} then uses these preliminaries to prove the main results in the manuscript. Constants are denoted $C, C'$, etc. They are positive, finite and may change from line to line.
 
\section{A handle on the cluster distribution}
\label{s:Reduction}
As mentioned in the introduction, understanding the cluster distribution $\nu$ is key to proving the statements in the manuscript. As in~\cite{CHL17}, a handle on this distribution is obtained by first identifying $\nu$ with the law of the cluster around a distinguished spine particle, conditioned to be the global maximum of the process. Then, by tracing the trajectory of the spine particle backwards in time, the events involved can be recast in terms of a decorated random-walk-like process conditioned to stay negative. The two reduction steps are summarized in the next two subsections. Estimates for such random walks are given in the succeeding subsection. Finally, some upper bounds from~\cite{CHL17}, derived using the statements in the first three subsections, are stated in the last subsection. For all proofs see~\cite{CHL17}. 
\subsection{Reduction to the cluster of the spine, conditioned to be the maximum}
\label{ss:Spine}
We begin by recalling the useful technique of spinal decomposition (c.f.,~\cite{HarRob2011}). The (one) spine branching Brownian motion (SBBM) is defined as the original process $(h,T,\rmd)$, only that at any given time one of the particles is designated as the {\em spine particle}. Particles which are not the spine branch and diffuse exactly as before. The spine particle also diffuses as before, but branches (by two) at rate $2$ not $1$. When the spine branches, one of its children, chosen uniformly at random, is designated the new spine. We shall use the same notation $(T,h, \rmd)$ for the SBBM process and distinguish this process from the original one by renaming the underlying probability measure to $\wt{\bbP}$ (with $\wt{\bbE}$ the corresponding expectation). The identity of the spine at time $t \geq 0$ will be recorded via the random variable $X_t \in L_t$. The genealogical line of decent of the spine particle, namely the function $t \mapsto X_t$, will be referred to as the spine of the process. 

The following is known as the Many-To-One Lemma. To avoid integrability issues, we state it for bounded functions. Recall that $\cF_t$ is the sigma-algebra generated by $h_t$, $T_t$ and $\rmd_t$ (but not $X_t$).
\begin{lem}[Many-To-One Lemma]
\label{l:4.1}
Let $F = (F(x) :\: x \in L_t)$ be a bounded $\cF_t$-measurable real-valued random function on $L_t$. Then, 
\begin{equation}
\bbE \Big( \sum_{x \in L_t} F(x) \Big)
= \rme^{t} \,\wt{\bbE} F(X_t) \,.
\end{equation}
\end{lem}
Recalling~\eqref{e:5B}, let now $\cC_{t,r}^* := \cC_{t, r_t} (X_t)$ denote the cluster around the spine. Thanks to Lemma~\ref{l:4.1} and the convergence in~\eqref{e:N7}, it is then not difficult to show,
\begin{lem}[Lemma 5.1 in~\cite{CHL17}]
\label{l:7.0}
Let $\cC \sim \nu$ be distributed according to the cluster law. Then for any $\nu$-continuity set $B \subseteq \bbM$,
\begin{equation}
\label{e:387}
\bbP(\cC \in B) = \lim_{t \to \infty} 
\wt{\bbP} \big( \cC_{t, r_t}^* \in B \, \big|\, \wh{h}_t(X_t) = \wh{h}_t^* = 0\big) \,.
\end{equation}
\end{lem}

\subsection{Reduction to a decorated random-walk conditioned to stay negative}
\label{ss:2.2}
Next, we introduce the decorated random-walk-like process using-which, probabilities such as the ones on the right hand side of~\eqref{e:387}, can be handled. Let $W=(W_s :\: s \geq 0)$ be a standard Brownian motion, whose initial position we leave free to be determined according to the conditional statements we make. For $0 \leq s \leq t$, we fix
\begin{equation}
\label{e:20.5}
\gamma_{t,s} := \tfrac{3}{2\sqrt{2}} \big(\log^+ s - \tfrac{s}{t}\log^+ t \big) \quad \text{and}  \qquad
\wh{W}_{t,s} := W_s - \gamma_{t,s} \,.
\end{equation}
Let us also define the collection $H=\big(h^s :\: s \geq 0\big)$ of independent copies of $h$, that we will assume to be independent of $W$ as well. Finally, let $\cN$ be a Poisson point process with intensity $2 \rmd x$ on $\bbR_+$, independent of $H$ and $W$ and denote by $\sigma_1 < \sigma_2 < \dots$ its ordered atoms. The triplet $(\wh{W}, \cN, H)$ forms what we shall call a {\em decorated random-walk-like process} (DRW). The underlying probability measure will still be denoted by $\bbP$ and the corresponding expectation by $\bbE$.

The following reduction statements appear in Subsection~3.1 of~\cite{CHL17}. We refer the reader to that reference for their straightforward proofs. Recall that $\rmB_{r}(x)$ is the ball of radius $r$ around $x$ in the genealogical distance $\rmd$, and that we write $\wh{h}_t = h_t -m_t$ and $\wh{h}^*_t = \max_{x \in L_t} \wh{h}_t(x)$. For $A \subseteq L_t$ set also $\wh{h}_t^*(A)$ for $\max_{x \in A} \wh{h}_t(x)$. The first lemma concerns the event that particles in a genealogical neighborhood of the spine (including the spine itself) stay below a given height. 
\begin{lem}
\label{l:5.1}
For all $0 \leq r \leq t$ and $u, w \in \bbR$,
\begin{equation}
\nonumber
\wt{\bbP} \Big( \wh{h}_t^*\big(\rmB^\rmc_{r}(X_t)\big)
 \leq u \, \Big| \, \wh{h}_t(X_t) = w \Big) \\
= \bbP \Big( \max_{k: \sigma_k \in [r,t]} \big(\wh{W}_{t,\sigma_k} + \wh{h}^{\sigma_k*}_{\sigma_k}\big) \leq 0 \ \Big|\ 
	\wh{W}_{t,r} = w - u,\, \wh{W}_{t,t} = -u \Big) \,.
\end{equation}
In particular for all $t \geq 0$ and $v,w \in \bbR$,
\begin{equation}
\nonumber
\wt{\bbP} \big( \wh{h}^*_t \leq u \, \big| \, \wh{h}_t(X_t) = w \big) 
= \bbP \Big( \max_{k: \sigma_k \in [0,t]} \big(\wh{W}_{t,\sigma_k} + \wh{h}^{\sigma_k*}_{\sigma_k}\big) \leq 0 \ \Big|\ 
	\wh{W}_{t,0} = w - u,\, \wh{W}_{t,t} = -u \Big) \,.
\end{equation}
\end{lem}

In a similar way, we can express the distribution of the cluster around the spine particle, given that it reaches height $m_t$. For what follows $\cE_t^s$ denotes the extremal process of $h^s_t$, defined as $\cE_t$ in~\eqref{e:O.2} only with respect to $h^s_t$ in place of $h_t$. 
Then,
\begin{lem}
\label{l:5.2}
For all $0 \leq r \leq t$,
\begin{multline}
\label{e:26.6}
\wt{\bbP}  \Big( \big(\cC_{t,r} (X_t), (h_{t-s}(X_{t-s}) - m_t)_{s \leq r} \big) 
\! \in\! \cdot \,\Big|\, \wh{h}^*_t = \wh{h}_t(X_t) = 0 \Big)  \\
= \! \bbP \bigg( \! 
\Big( \! \sum\limits_{\sigma_k \leq r } \! \cE_{\sigma_k}^{\sigma_k} \big(\cdot - \wh{W}_{t,\sigma_k}\big),\, 
(\wh{W}_{t,s} \!- \! m_s)_{s \leq r }\! \Big)\! \in \! \cdot 
	 \Big| \max_{k: \sigma_k \in [0,t]} \big(\wh{W}_{t,\sigma_k} \!+\! \wh{h}^{\sigma_k*}_{\sigma_k}\big) \leq 0 , 
	 	\wh{W}_{t,0} = \wh{W}_{t,t} = 0 \! \bigg)
\end{multline}
\end{lem}

\subsection{Estimates for a decorated random-walk conditioned to stay negative}
\label{ss:2.3}
Having reduced the analysis of the cluster distribution to the study of the DRW conditioned to stay positive, we need estimates for probabilities involving the latter. Such probabilities are treated in general in \cite{CHL17Supp}, which is the supplement material to~\cite{CHL17}. Specialized to our setup here, they take the form of the two lemmas below. The first lemma was already included in~\cite{CHL17} (see third part of Lemma~3.4).
\begin{lem}[Lemma 3.4 in~\cite{CHL17}]
\label{lem:15}
There exists non-increasing functions $g: \bbR \to (0,\infty)$ and $f: \bbR \to (0,\infty)$ such that
\begin{equation}
\label{e:53}
\bbP \Big( \max_{k: \sigma_k \in [0,t]} \big(\wh{W}_{t,\sigma_k} + \wh{h}^{\sigma_k*}_{\sigma_k}\big) \leq 0 
	\,\Big|\, \wh{W}_{t,0} = v \,,\,\, \wh{W}_{t,t} = w \Big) 
	\sim 2 \frac{f(v) g(w)}{t},
\end{equation}
as $t \to \infty$ uniformly in $v,w$ satisfying $v,w < 1/\epsilon$ and $(v^-+1)(w^-+1) \leq t^{1-\epsilon}$ for any fixed $\epsilon > 0$. 
\end{lem}

Next we have the following sharp entropic-repulsion result. This was not needed in~\cite{CHL17} and hence left out from that work. Nevertheless, it is an immediate consequence of Proposition 1.5 from the supplement material~\cite{CHL17Supp} of~\cite{CHL17}.
\begin{lem}
\label{l:5.4}
For all $M \geq 1$,
\begin{equation}
\label{e:652}
\sup_{s \geq 1} \limsup_{t \to \infty}\, \sqrt{s} \,\bbP \Big( \max_{u \in [s, t-s]} \wh{W}_{t,u}  \geq -M
	\,\Big|\,  \max_{\sigma_k \in [0,t]} \big(\wh{W}_{t,\sigma_k} \! + \! \wh{h}^{\sigma_k*}_{\sigma_k}\big) \leq 0  \,,\, \wh{W}_{t,0} \!= \! \wh{W}_{t,t} \! = \! 0 \Big) 
	\!< \! \infty \,.
\end{equation}
\end{lem}

\begin{proof}
Recalling the definition of $\wh{W}_{t,u}$ the desired statement is precisely that in Proposition~1.5 in~\cite{CHL17Supp} with $x=y=0$. We just need to verify that the assumptions in that proposition hold. Assumptions~(1.1),~(1.2) and~(1.3) from~\cite{CHL17Supp} (with $\lambda = 2$ and sufficiently small $\delta >0$) were already verified in the proof of Lemma~\ref{lem:15} in~\cite{CHL17}. It remains therefore to check that
\begin{equation}
\begin{split}
\label{e:A3s}
\gamma_{t,u} - \frac{u}{r} \gamma_{t,r} & \leq \delta^{-1} \big(1+(\wedge^r(u))^{1/2-\delta}\big)
\\
\gamma_{t,u'} - \frac{t-u'}{t-r} \gamma_{t,r} & \leq \delta^{-1} \big(1+(\wedge^{t-r}(u'-r))^{1/2-\delta}\big)\,,
\end{split}
\end{equation}
for all $0 < u < r < u' < t$ and some $\delta > 0$. To this end, a simple computation gives
\begin{equation}
\begin{split}
\gamma_{t,u}  - \gamma_{t,r}(u/r) & = \tfrac{3}{2\sqrt{2}} \Big(\log^+r - \tfrac{r}{u} \log^+ u\Big) \\
\gamma_{t,u'} - \tfrac{t-u'}{t-r}\gamma_{t,r}  & = \tfrac{3}{2\sqrt{2}} \Big( \log^+ u - \Big(\tfrac{t-u'}{t-r} \log^+ r + \tfrac{u'-r}{t-r} \log^+ t \Big)\Big) \,.
\end{split}
\end{equation}
Then the upper bound of Lemma~3.3 in~\cite{CHL17}, once with $(u,r,0)$ and once with $(t, u'-r, r)$ in place of $(t,s,r)$, gives~\eqref{e:A3s} with any $\delta \leq 2\sqrt{\log 2}/3$.
\end{proof}

\subsection{Moment bounds for the number of cluster points}
Next we state several moment bounds for the number of cluster points. These were derived in~\cite{CHL17} using the statements in the last three subsections and some non-negligible work. We start with a second moment bound for $\cC([-v, 0])$. It is part of Proposition~1.5 in~\cite{CHL17}.
\begin{lem}
\label{l:12}
Let $\cC \sim \nu$. Then there exists $C > 0$ such that for all $v \geq 0$,
\begin{equation}
\label{e:28}
\bbE \big[ \cC([-v, 0])^2 \big] \leq C (v+1) \rme^{2\sqrt{2} v} \,.
\end{equation} 
\end{lem}
\noindent
Next, if $v \geq 0$ and $0 \leq s \leq t$ we define
\begin{equation}
	\label{e:394}
J_{t,v}^{\geq M}(s):= \cE_s^s \big([-v,0] - \wh{W}_{t,s}\big) \times 1_{\{|\wh{W}_{t,s}| \geq M\}} \times 
1_{\{\wh{h}^{s*}_s \leq -\wh{W}_{t,s}\}}  \times 1_{\{\max_{k : \sigma_k \in [0,t]} \big(\wh{W}_{t,{\sigma_k}} + \wh{h}^{\sigma_k*}_{\sigma_k}\big) \leq 0 \}}
\end{equation}
and set
\begin{equation}
j_{t,v}^{\geq M}(s) := \bbE \Big( J_{t,v}^{\geq M}(s) \, \Big|\,  \wh{W}_{t,0} = \wh{W}_{t,t} = 0 \Big) \,.
\end{equation}
We also abbreviate $J_{t,v}(s) \equiv J_{t,v}^{\geq 0}(s)$ and $j_{t,v}(s) \equiv j_{t,v}^{\geq 0}(s)$. The following is part of Lemma~5.5 in~\cite{CHL17}.
\begin{lem}
\label{l:7.4}
There exists $C, C' > 0$ such that for all $t \geq 0$, $0 \leq s \leq t/2$, $v \geq 0$ and $M \geq 0$,
\begin{equation}
\label{e:603}
j_{t,v}^{\geq M}(s) \leq C \frac{  \rme^{\sqrt{2} v} (v+1) }{t (s+1) \sqrt{s}} \times  \rme^{-C' M} \Big(\rme^{- \frac{v^2}{16s}} + \rme^{-\frac{v}{2}}\Big) \,.
\end{equation}
\end{lem}

\section{Proof of Main Results}
\label{s:3}
\subsection{Proof of Proposition~\ref{p:2.3}}
The proof of Proposition~\ref{p:2.3} will be based on the following two lemmas. 
\begin{lem}
\label{l:7.6}
The following quantity
\begin{equation}
\label{e:401}
\rme^{-\sqrt{2} v} 
\wt{\bbE}\Big(\cC_{t,r_t}^*\big([-v,0]\big) ;\; \max_{s\in[\eta v^2, \eta^{-1} v^2]} 	\big(h_{t-s}(X_{t-s}) - m_t + m_s\big) \leq -M \,\Big|\, \wh{h}^*_t = \wh{h}_t(X_t) = 0\Big) \
\end{equation}
tends to $0$ when $t \to \infty$, followed by $v \to \infty$, then $M \to \infty$ and finally $\eta \to 0$.
\end{lem}
\begin{lem}
\label{l:7.7}
For all $\eta > 0$ and $M > 0$, there exists $\delta > 0$ such that,
\begin{equation}
\nonumber
\wt{\bbP} \Big(\max_{s\in[\eta v^2, \eta^{-1} v^2]} 
	\big(h_{t-s}(X_{t-s}) - m_t + m_s\big) > -M \,\Big|\, \wh{h}^*_t = \wh{h}_t(X_t) = 0 \Big)  \leq \frac{\delta^{-1}}{v} \,. 
\end{equation}
for all $v$ large enough and then $t$ large enough.
\end{lem}

Let us first prove Proposition~\ref{p:2.3}.
\begin{proof}[Proof of Proposition~\ref{p:2.3}]
The second statement in the proposition follows trivially from~\eqref{e:29} and Markov's inequality. It therefore remains to show~\eqref{e:33.2}. Given $t \geq 0$, $v \geq 0$, $M > 0$ and $\eta > 0$, define the event
\begin{equation}
\cB :=  \Big\{ \max_{s\in[\eta v^2, \eta^{-1} v^2]} 
	\big(h_{t-s}(X_{t-s}) - m_t + m_s\big) > -M \Big\} \,.
\end{equation}
Thanks to Lemma~\ref{l:7.6} and Lemma~\ref{l:7.7}, for any $\epsilon > 0$ there exist $\eta > 0$, $M > 0$ and $\delta' > 0$ such that for all $v$ and then $t$ large enough,
\begin{equation}
\wt{\bbE} \Big(\cC_{t,r_t}^*\big([-v,0]\big) ;\; \cB^\rmc \, \Big| \, \wh{h}^*_t = \wh{h}_t(X_t) = 0 \Big)
\leq \frac{\epsilon}{2} \, \rme^{\sqrt{2} v} 
\quad, \qquad
\wt{\bbP} \Big(\cB\,\Big|\, \wh{h}^*_t = \wh{h}_t(X_t) = 0\Big) \leq \frac{1}{2\delta' v} \,. 
\end{equation}
Next let $\delta := \epsilon \delta'$ and define also the event
$\cA: = \big\{\cC_{t,r_t}^*\big([-v,0]\big) > \delta v \rme^{\sqrt{2}v} \big\}$.
Then,
\begin{equation}
\nonumber
\begin{split}
\wt{\bbE} \Big(\cC_{t,r_t}^* &\big([-v,0] \big)  ;\; \cA^\rmc \, \Big| \, \wh{h}^*_t = \wh{h}_t(X_t) = 0  \Big) \\ & \leq 
\wt{\bbE} \Big(\cC_{t,r_t}^*\big([-v,0]\big) ;\; \cB^\rmc \, \Big| \wh{h}^*_t = \wh{h}_t(X_t) = 0  \Big)
+
\wt{\bbE} \Big(\cC_{t,r_t}^*\big([-v,0]\big) ;\; \cB \cap \cA^\rmc \, \Big| \, \wh{h}^*_t = \wh{h}_t(X_t) = 0 \Big) 
\\ &\leq
 \frac{\epsilon \rme^{\sqrt{2} v}}{2}  + \frac{1}{2 \delta' v} \delta'\epsilon   v \rme^{\sqrt{2} v} 
\leq \epsilon \rme^{\sqrt{2} v} \,,
\end{split}
\end{equation}
Now take $t \to \infty$ in the last display. Then by Lemma~\ref{l:7.0} and the bounded convergence theorem,~\eqref{e:33.2} follows for all large enough $v$ such that $\nu$-almost-surely the point $-v$ is not charged by $\cC$. Removing this stochastic continuity restriction requires a standard argument, the kind of which was used in many of the proofs in~\cite{CHL17} (e.g., the proof of Proposition~1.5). We therefore omit further details.
\end{proof}

\begin{proof}[Proof of Lemma~\ref{l:7.6}]
Using Lemma~\ref{l:5.1} and Lemma~\ref{l:5.2} and abbreviating
\begin{equation}
\cR_t \equiv \Big\{ \max_{k : \sigma_k \in [0,t]} \big(\wh{W}_{t,{\sigma_k}} + \wh{h}^{\sigma_k*}_{\sigma_k}\big) \leq 0 \,,\, \wh{W}_{t,0} = \wh{W}_{t,t} = 0 \Big\} \,,
\end{equation}
we can write the product of the expectation in~\eqref{e:401} with $\wt{\bbP}\big(\wh{h}^*_t \leq 0\,\big|\, \wh{h}_t(X_t) = 0 \big)$ as
 \begin{multline}
\label{e:3.2}
\bbE \Big(\int_{s=0}^{r_t} J_{t,v}(s) \cN(\rmd s) \;;\; \max_{s \in [\eta v^2, \eta^{-1} v^2]} 
	\wh{W}_{t,s} \leq -M \,\Big|\, \cR_t \Big)  \\
\leq 
\bbE \Big(\int_{s=0}^{r_t} \big(J_{t,v}(s) 1_{[\eta v^2, \eta^{-1} v^2 ]^\rmc}(s)
+ J_{t,v}^{\geq M}(s) 1_{[\eta v^2, \eta^{-1} v^2]} \big)\cN(\rmd s) \,\Big|\, \cR_t \Big) \,,
\end{multline}
where $J_{t,v}^{\geq M}(s)$ is as in~\eqref{e:394}. Then, by the Palm-Campbell Theorem, the right hand side above is equal to
\begin{equation}
\label{e:3.3}
2 \int_{s=0}^{r_t} \Big(j_{t,v}(s) 1_{[\eta v^2,  \eta^{-1}v^2 ]^\rmc} +
j^{\geq M}_{t,v}(s) 1_{[\eta v^2,  \eta^{-1} v^2]} \Big) \rmd s \,.
\end{equation}
Using Lemma~\ref{l:7.4} we can upper bound the above integral by $C t^{-1} \rme^{\sqrt{2} v} (v+1)$ times
\begin{equation}
\label{e:924}
\begin{split}
\int_{s=0}^\infty & \frac{\rme^{-v^2/(16s)}  + \rme^{-v/2}}{ \sqrt{s} (s+1)}  
\Big(1_{ \{ s \in [\eta v^2, \eta^{-1}v^2]^\rmc \}} + \rme^{-C' M} 1_{\{ s \in [ \eta v^2, \eta^{-1}v^2] \}} \Big) \rmd s \\
& \ 
\leq \int_{0}^\infty  \! \frac{\rme^{-\frac{v}{2}}}{ \sqrt{s} (s\!+\!1)}  \rmd s 
\! +\! \int_{0}^{\eta v^2} \!  \frac{\rme^{-v^/(16s)}}{ s^{3/2}}  \rmd s  
\!+ \! \rme^{-C'M} \! \int_{\eta v^2}^\infty  s^{-\frac{3}{2}}  \rmd s 
\!+\! \int_{\eta^{-1} v^2}^\infty  s^{-\frac{3}{2}}  \rmd s  \\
& \ \leq 
C \bigg(\rme^{-\frac{v}{2}} + \frac{1}{v}  + \frac{\rme^{-C' M}}{v  \sqrt{\eta} } + \frac{ \sqrt{\eta}}{v} \,  \bigg) \,.
\end{split}
\end{equation}
Altogether~\eqref{e:3.3} is at most  $C t^{-1} \rme^{\sqrt{2}v} \big(\rme^{-v/4} + \rme^{-C' M}\eta^{-\frac{1}{2}} + \eta^{\frac{1}{2}}\big)$. Since $\wt{\bbP}\big(\wh{h}^*_t \leq 0\,\big|\, \wh{h}_t(X_t) = 0 \big) \sim Ct^{-1}$ as $t \to \infty$, thanks to the second part of Lemma~\ref{l:5.1} with $u=w=0$ and Lemma~\ref{lem:15}, the desired statement follows.
\end{proof}

\begin{proof}[Proof of Lemma~\ref{l:7.7}]
Thanks to Lemma~\ref{l:5.2} (ignoring the distribution of $\cC_{t,r}(X_t)$), the conditional probability in the statement of the lemma is equal to 
\begin{equation}
\nonumber
\bbP \Big( \max_{s \in [\eta v^2, \eta^{-1} v^2]} \wh{W}_{t,s} > -M  \,\Big|\,  	
\max_{k : \sigma_k \in [0,t]} \big(\wh{W}_{t,{\sigma_k}} + \wh{h}^{\sigma_k*}_{\sigma_k}\big) \leq 0 \,,\, \wh{W}_{t,0} = \wh{W}_{t,t} = 0 \Big) \,.
\end{equation}
Invoking Lemma~\ref{l:5.4}, then gives the desired upper bound with some $\delta > 0$, depending on $\eta$, $M$ and all $v$ and then $t$ large enough.
\end{proof}

\subsection{Proofs of Theorem~\ref{t:15} and Corollary~\ref{c:1.3}}
The proofs of Theorem~\ref{t:15} will be based on the following lemma.
\begin{lem}
\label{l:15}
For any $\epsilon > 0$, there exists $\delta > 0$ such that for all $\alpha \in (0,1)$,
\begin{equation}
\label{e:59}
	\lim_{v\to\infty}\mathbb{P} \left(
	\frac{\cE \big([-v, \infty) ;\;  [-\alpha v, \infty) \times \bbM ,\,  F_\delta(-v)^\rmc \big)}
	{C_\star Z \alpha v \rme^{\sqrt{2} v} } \leq \epsilon \right) = 1 \,.
\end{equation}
Moreover, there exists $C>0$ such that for any $\delta > 0$ and $\alpha \in (0,1)$,
\begin{equation}
\label{e:58}
	\lim_{v\to\infty}\mathbb{P}\left(\frac{\wh{\cE}\big([-\alpha v, \infty) \times \bbM   \cap F_\delta(-v) \big)}
		{Z\rme^{\sqrt{2} \alpha v}/\sqrt{2}} <  \frac{C} {\delta (1-\alpha)} v^{-1} \right)=1\,.
\end{equation}
\end{lem}
\begin{proof}
Let $\alpha \in (0,1)$. Given $-\infty < -v < w < z \leq \infty$ and $\delta > 0$, define
\begin{equation}
\label{e:43}
\textstyle
H_v (w,z) := \sum_{(u, \cC) \in \wh{\cE}} \ \cC \big( [-v - u, 0] \big) \times 1_{[w, z]}(u)  \,.
\end{equation}
and
\begin{equation}
\label{e:60.1}
Q_{v, \delta}(w, z) :=
\sum_{(u, \cC) \in \wh{\cE}} \cC \big([-v-u, 0] \big) \times 1_{[w,z)}(u) 
	\times 1_{\{\cC ( [-v-u, 0] ) \leq 
		\delta (v+u) \rme^{\sqrt{2}(v+u)}\}} 
 \,.
\end{equation}
Clearly $Q_{v, \delta}(w, z) \leq H_v(w,z)$. Therefore, setting $z:= \sqrt{\log v}$ and $w:= -\alpha v$, we can bound the numerator in~\eqref{e:59} by $Q_{v, \delta} (w, z) + H_v (z, \infty)$.

Since conditional on $Z$, the intensity measure governing the law of $\wh{\cE}$ is finite on $[0,\infty) \times \bbM$ almost surely, we must have $\wh{\cE}([z, \infty) \times \bbM) = 0$  for all large enough $z$. This shows that 
\begin{equation}
\label{e:4.9}
H_v(z,\infty) \big(Zv\rme^{\sqrt{2}v}\big)^{-1} \longrightarrow 0
\quad \text{as } v \to \infty \text{ a.s.}
\end{equation}
At the same time, for any $\epsilon > 0$, thanks to the first part of Proposition~\ref{p:2.3}, we may find $\delta >0$ such that
\begin{equation}
\nonumber
\begin{split}
\bbE \big( Q_{v, \delta} (w, z)\,\big|\, Z \big) & = 
\int_{u=w}^z 
	\bbE \Big( \cC([-v-u, 0]) ;\; \cC([-v-u, 0]) \leq \delta (v+u) \rme^{\sqrt{2}(v+u)} \Big) Z \rme^{-\sqrt{2} u} \rmd u \\
& \leq \int_{u=w}^z 
	\epsilon \rme^{\sqrt{2}(v+u)} Z \rme^{-\sqrt{2} u} \rmd u 
\leq 2\epsilon Z \alpha v \rme^{\sqrt{2}v} \,.
\end{split}	
\end{equation}
In addition, thanks to Lemma~\ref{l:12},
\begin{equation}
\nonumber
\begin{split}
\Var \big( Q_{v, \delta} (w, z)\,\big|\, Z \big) & =
\int_{u=w}^z 
	\bbE \Big( \cC([-v-u, 0]) ;\; \cC([-v-u, 0]) > \delta (v+u) \rme^{\sqrt{2}(v+u)}\Big)^2 Z 
	\rme^{-\sqrt{2} u} \rmd u \\
& \leq 
\int_{u=w}^z 
	\bbE \Big( \cC([-v-u, 0])\Big)^2 Z \rme^{-\sqrt{2} u} \rmd u 
\, \leq\, C Z v \rme^{2 \sqrt{2} v + \sqrt{2} z} \,.
\end{split}	
\end{equation}

It follows by Chebyshev's inequality that
\begin{equation}
\bbP \Big(Q_v(w,z) > 4 \epsilon Z \alpha v \rme^{\sqrt{2}v}\, \Big|\, Z \Big) 
\overset{v \to \infty} \longrightarrow 0 \,,
\end{equation}
almost surely. By the bounded convergence theorem, the above limit holds also for the unconditional probability. Together with~\eqref{e:4.9} and the union bound, this shows that as $v \to \infty$,
\begin{equation}
\bbP \Big(Q_v(w,z) + H_v(z, \infty) > 8 \epsilon Z \alpha v \rme^{\sqrt{2}v} \Big) 
\longrightarrow 0 \,.
\end{equation}
Since the quantity on the left hand side in the probability above dominates the one in the numerator of~\eqref{e:59}, this gives~\eqref{e:59} with $8\epsilon/C_\star$ in place of $\epsilon$. 

Turning to the second statement of the lemma, by the definition of $\wh{\cE}$, for any $\delta > 0$ and $\alpha \in (0,1)$, the law of the first numerator in~\eqref{e:58}, conditional on $Z$, is Poisson with parameter given by
\begin{equation}
\label{e:153}
\lambda_{\alpha, \delta}(v)=\int_{u=-\alpha v}^\infty 
	\bbP \Big(\cC([-v-u, 0]) > \delta (v+u) \rme^{\sqrt{2}(v+u)}\Big) Z \rme^{-\sqrt{2} u} \rmd u \,.
\end{equation}
Therefore by the second part of Proposition~\ref{p:2.3},
\begin{equation}
\label{e:64.1.1}
\lambda_{\alpha, \delta}(v) \leq \int_{u=-\alpha v}^\infty 
	\frac{C}{\delta(v+u)} Z \rme^{-\sqrt{2} u} \rmd u 
\leq \frac{C'}{\delta (1-\alpha)} v^{-1}Z\rme^{\sqrt{2} \alpha v} /\sqrt{2} \,,
\end{equation}
for some $C' > 0$.
Then by Chebyshev's inequality, conditional on $Z$, the probability of the complement of the event in~\eqref{e:58} with $C := 2C'$ is at most
\begin{equation}
\frac{\lambda_{\alpha, \delta}(v)}{\left(\frac{2C'}{\delta(1-\alpha)}v^{-1} Z  \rme^{\sqrt{2} \alpha v} /\sqrt{2} -\lambda_{\alpha, \delta}(v)\right)^2} \,,
\end{equation}
which goes to $0$ as $v \to \infty$, for $\bbP$-almost every $Z$. The same then also holds for the unconditional probability, thanks again to the bounded convergence theorem.
\end{proof}

We can now prove Theorem~\ref{t:15}.
\begin{proof}[Proof of Theorem~\ref{t:15}]
Given $\epsilon > 0$, we use Lemma~\ref{l:15} to find $\delta > 0$ such that for any $\alpha \in (0,1)$ \eqref{e:59} holds. We then use~\eqref{e:1.8},~\eqref{e:1.10} and~\eqref{e:37} to claim that also
\begin{equation}
\label{e:402}
\frac{\cE \big([-v,\infty) ;\; [-\alpha v, \infty) \times \bbM \big)}{C_\star Z \alpha v\rme^{\sqrt{2}v}} > \frac{3}{4}
 \quad \text{ and } \quad
\frac{\cE^*\big([-\alpha v,\infty))}{Z \rme^{\sqrt{2} \alpha v} /\sqrt{2}} > \frac{3}{4}
 \end{equation}
holds with probability tending to $1$ as $v \to \infty$.

Dividing both the numerator and denominator of~\eqref{e:159} by $C_\star Z \alpha v \rme^{\sqrt{2}v}$, we now observe that whenever the event in~\eqref{e:59} and the first event in~\eqref{e:402} hold, we must also have~\eqref{e:159} with $4\epsilon/3$ in place of $\epsilon$.
At the same time, observing that $\cE^*\big([-\alpha v,\infty) ;\; F_\delta(-v)\big) = 
\wh{\cE}\big(\big([-\alpha v, \infty) \times \bbM\big) \cap F_{\delta}(-v)\big)$ which follows by definition, we now divide both the numerator and denominator of~\eqref{e:158} by $Z \rme^{\sqrt{2} \alpha v}/\sqrt{2}$. We then see that whenever the event in~\eqref{e:58} and the second event in~\eqref{e:402} hold, we must also have~\eqref{e:158} with $4C/3$ in place of $C$. Renaming $\epsilon$ and $C$, and using the union bound, we complete the proof of the theorem for $\cE$ and $\cE^*$.

To obtain the finite $t$ analogs (with $t \to \infty$), it is sufficient to argue that all random quantities on the left hand sides of~\eqref{e:159} and~\eqref{e:158} are the joint weak limits of their respective finite time analogs as $t \to \infty$. This in turn follows from~\eqref{e:N7} and standard arguments, the kind of which were used in many of the proofs in~\cite{CHL17} (e.g., the proof of Theorem~1.1). We therefore omit further details.
\end{proof}

\begin{proof}[Proof of Corollary~\ref{c:1.3}]
We shall only show that statement for $\cE^*$ and $\cE$ as the argument for the case involving $\cE_t$ and $\cE^*_t$ is almost identical. Given $\epsilon > 0$, we let $\delta$ be given by Theorem~\ref{t:15} for $\epsilon/2$, choose $\alpha: = 1-\epsilon/3$ and set 
\begin{equation}
G_\epsilon(-v) := \big([-\alpha v, \infty) \times \bbM\big) \cap  F_\delta(-v) \,.
\end{equation}
Then, the first part of Theorem~\ref{t:15}, together with~\eqref{e:37} and the union bound show that 
\begin{equation}
\cE \big([-v, \infty) ;\;  G_\epsilon(-v) \big) 
\, \geq \, (1-\epsilon/2) 
\cE \big([-v, \infty) ;\;  [-\alpha v, \infty) \times \bbM \big) 
\, \geq \, (1-\epsilon)
\cE \big([-v, \infty)\big) \,,
\end{equation}
with probability at least $1-\epsilon/2$ whenever $v$ is large enough.
At the same time, the second part of Theorem~\ref{t:15} shows that 
\begin{equation}
\cE^*\big([-v, \infty) ;\; G_\epsilon(-v)\big) \leq \epsilon\, \cE^*\big([-\alpha v, \infty)\big)
< \epsilon\, \cE^*\big([-v, \infty)\big) \,,
\end{equation}
with probability at least $1-\epsilon/2$, again whenever $v$ is large enough. A final application of the union bound then completes the proof.
\end{proof}

\section*{Acknowledgments} 
The work of A.C. was supported by the Swiss National Science Foundation~200021\underline{{ }{ }}163170. The work of O.L. was supported by the Israeli Science Foundation grant no. 1382/17 and by the German-Israeli Foundation for Scientific Research and Development grant no. I-2494-304.6/2017.

\bibliographystyle{abbrv}
\bibliography{GBBStructure}

\end{document}